\documentclass[12pt]{amsart}
\usepackage{amssymb,amsmath}
\usepackage{enumerate}
\usepackage{mathtools}
\usepackage{color}
\usepackage{mathrsfs}
\newtheorem{theorem}{Theorem}[section]
\newtheorem{proposition}[theorem]{Proposition}
\newtheorem{lemma}[theorem]{Lemma}

\theoremstyle{definition}

\newtheorem{example}[theorem]{Example}

\newcommand{\clb}{\mathscr{B}}
\newcommand{\clh}{\mathcal{H}}

\title{Nayak's theorem for compact operators}

\author{B V Rajarama Bhat}
\address{Indian Statistical Institute, Bangalore}
\email{bvrajaramabhat@gmail.com}
\author{ Neeru Bala}
\address{Indian Institute of Technology (Indian School of Mines),
Dhanbad} \email{neerusingh41@gmail.com}

\begin{document}
\begin{abstract}
Let $A$ be an $m\times m$ complex matrix and let $\lambda _1,
\lambda _2, \ldots , \lambda _m$ be the eigenvalues of $A$ arranged
such that $|\lambda _1|\geq |\lambda _2|\geq \cdots \geq |\lambda
_m|$ and for $n\geq 1,$ let $s^{(n)}_1\geq s^{(n)}_2\geq \cdots \geq
s^{(n)}_m$ be the singular values of $A^n$. Then a famous theorem of
Yamamoto (1967) states that
$$\lim _{n\to \infty}(s^{(n)}_j )^{\frac{1}{n}}= |\lambda _j|, ~~\forall \,1\leq j\leq
m.$$ Recently S. Nayak strengthened this result very significantly
by showing that the sequence of matrices $|A^n|^{\frac{1}{n}}$
itself  converges to a positive matrix $B$ whose
eigenvalues are $|\lambda _1|,|\lambda _2|,$ $\ldots , |\lambda
_m|.$  Here this theorem has been extended  to arbitrary compact
operators on infinite dimensional complex separable Hilbert spaces.
The proof makes use of  Nayak's theorem, Stone-Weirstrass theorem,
Borel-Caratheodory theorem and some technical results of Anselone
and Palmer on collectively compact operators. Simple examples show
that the result does not hold for general bounded operators.

\end{abstract}
\maketitle
\section{Introduction}

The main aim of this paper is to study convergence of a sequence of
positive bounded linear operators of the form $(|T^n|^{1/n})$, where
$T$ is a compact operator on a separable Hilbert space $\clh$.
Sequences of this form are of essential importance in spectral
theory
as seen from the classical spectral radius formula. % In the following
%$\mathscr{B}(\mathcal{H})$ denotes the algebra of all bounded
%operators on a complex separable Hilbert space $\mathcal{H}.$

 For $m\geq 1,$ any complex matrix $A=[a_{ij}]_{1\leq i,j\leq m}$
is considered as an operator on the finite dimensional Hilbert space
$\mathbb{C}^m$ in the usual way. By spectral radius formula we know
that $\lim _{n\to \infty}\|A^n\|^{\frac{1}{n}}$ exists and equals to
the spectral radius of $A$. A well-known result of Yamamoto goes
much further. Let $\lambda _1, \lambda _2, \ldots , \lambda _m$ be
the eigenvalues of $A$. We assume that $|\lambda _1|\geq |\lambda
_2|\geq \cdots \geq |\lambda _m|$.  For $n\geq 1,$ let
$s^{(n)}_1\geq s^{(n)}_2\geq \cdots \geq s^{(n)}_m$ be the singular
values of $A^n$. Then the main theorem of Yamamoto \cite{YAMA1}
states that
$$\lim _{n\to \infty}(s^{(n)}_j)^{\frac{1}{n}} = |\lambda _j|, ~~\forall\, 1\leq j\leq
m.$$ The spectral radius formula is the special case with $j=1.$ In
\cite{SOY}, S. Nayak proved a spatial version of Yamamoto's theorem.
In fact, he showed that the sequence of matrices
$(|A^n|)^{\frac{1}{n}}$ converges to a positive matrix $B$, whose
eigenvalues are $|\lambda _1|, |\lambda _2|, \ldots |\lambda_m|$.
From this Yamamoto's result follows as a simple corollary.

Yamamoto tried to extend his results to compact operators on
infinite dimensional separable Hilbert spaces and obtained some
partial results in \cite{YAMA2}. A satisfactory extension was got by
Davis \cite{CHAN}. Presently we extend Nayak's result to compact
operators and show that for any compact operator $A$,
$(|A^n|)^{\frac{1}{n}}$ converges in norm to a positive operator
$B$, whose spectrum is $\{|\lambda |:\lambda \in \sigma (A)\}.$ The
proof is considerably different from that of \cite[Theorem
3.8]{SOY}, for technical reasons.

Section 2 contains some introductory material and some basic tools
which we are going to need. One specific concept we will use is the
notion of collectively compact operators.

Section 3 contains the main result generalizing Nayak's theorem to
compact operators. We also provide some examples to show that the
result does not hold for general bounded operators even if relax the
convergence to strong or weak operator topology.

\section{Preliminaries}

Let $\mathscr{B}(\mathcal{H})$ be the space of all bounded operators
on a complex separable Hilbert space $\mathcal{H}.$ For any bounded
operator $T$, $\mbox{ran}(T)$ and $\mbox{ker}(T)$ will denote the
range and the kernel of $T$, respectively. For any vector space
$\mathcal{M}$, $\mbox{dim}(\mathcal{M})$ will denote the dimension
of $\mathcal{M}.$

To begin with we recall a case of Holder-McCarthy inequality.
 \begin{proposition}\cite[Theorem 1.4, Page 5]{FURUTA}\label{thm inequ furuta}
        Let $S\in\mathscr{B} (\clh)$ be a positive operator. Then
        \begin{align*}
            \langle S^rx,x\rangle\leq\langle Sx,x\rangle^r
        \end{align*}
        for all $0<r<1$ and every unit vector $x\in\clh$.
    \end{proposition}
As a  special case of the above result, we get that
    \begin{align}
        \langle|T^n|^{1/n}x,x\rangle\leq\langle |T^n|x,x\rangle^{1/n}
    \end{align}
    for every $T\in\clb(\clh)$ and unit vector $x\in\clh$.
This inequality plays a crucial role in the study of limits of
$(|T^n|^{1/n})$.

One of the key ingredient of our proof is the notion of collectively compact sets of operators. This has been extensively studied in the literature: See \cite{PALMER,PALMER1} for more details. %Using this we have proved the convergence of the sequence $(|T^n|^{1/n})$ in the norm topology.
We say, a set $\mathcal{K}\subseteq\clb(\clh)$ is  {\em collectively
compact\/}  if and only if the set $\{Kx:K\in\mathcal{K},\,\|x\|\leq
1\}$ has compact closure. Originally the notion of collectively
compact set of operators was introduced to study approximate
solution of integral equations in \cite{ANSELONE}. Various
properties of such sets are studied in \cite{ANSEL,ATKINSON2}. We
simply recall the results we are going to need.

\begin{proposition}\cite[Proposition 2.1]{PALMER}\label{Palmer 2.1}
    Let $T,T_n\in\clb(\clh)$ for $n\in\mathbb{N}$. Then the following are equivalent.
    \begin{enumerate}
        \item $(T_n)$ converges to $T$ in SOT and $(T_n)$ is collectively compact.
        \item $(T_n)$ converges to $T$ in SOT and $(T_n-T)$ is collectively compact, and $T$ is compact.
    \end{enumerate}
\end{proposition}
%In the following, we examine how the concept of collectively compact set of operators fits into our setup.

\begin{proposition}\cite[Proposition 2.2]{PALMER}\label{Palmer 2.2}
Let $T,T_n\in\clb(\clh)$ for $n\in\mathbb{N}$. If $(T_n)$ converges
to $T$ in norm topology and each $T_n-T$ is compact then $(T_n-T)$
is collectively compact.
\end{proposition}

%In the following we are quoting only a relevant part of
%\cite[Theorem 3.4]{PALMER}.
\begin{proposition}\cite[Theorem 3.4]{PALMER}\label{Palmer 2.3}
    Let $T,T_n\in\clb(\clh)$ for $n\in\mathbb{N}$. Suppose $(T_n)$ converges to $T$ in strong operator topology and $(T_n-T)$ is collectively
    compact. Then $(T_n)$ converges to
    $T$ in norm if one of the following holds.
    \begin{enumerate}
        \item $T_n$ is self-adjoint for every $n\in\mathbb{N}$.
        \item $T_n-T$ is normal for every $n\in\mathbb{N}$.
        \item  $\{T_n^*-T^*:\,n\in\mathbb{N}\}$ is a collectively compact set.
    \end{enumerate}
\end{proposition}

The following statement can be found in \cite{ANSEL}, just prior to
Theorem 2.1.

\begin{proposition}\label{Ane} \cite{ANSEL} Assume $\{K_{m,n}:n\geq \}$ is collectively compact for each $m\in \mathbb{N}$ and
$\{K_{m,n}\}$ converges in norm to some bounded operator $K_n$,
uniformly in $n$, as $m$ tends to $\infty .$ Then $\{K_n:n\geq 1\}$
is collectively compact.
\end{proposition}

%%%NEW

\begin{proposition}\label{proposition linear combination}
    Let $\{T_n:n\in\mathbb{N}\}\subset\clb(\clh)$ be a collectively compact set. If $A,B\in\clb(\clh)$ are compact operators, then $\{AT_n+B: n\in\mathbb{N}\}$ and $\{T_nA+B:n\in\mathbb{N}\}$ are collectively compact sets of operators.
\end{proposition}
\begin{proof}
Since $\{T_n\}$ is a collectively compact set,  we have $\{2AT_n\}$
and $\{2T_nA\}$ are collectively compact, by \cite[Proposition
2.3]{PALMER1}. We know that finite union of collectively compact
sets is collectively compact, thus $\{2AT_n,2B\}$ and $\{2T_nA,2B\}$
are collectively compact. By \cite[Proposition 2.1]{PALMER1}, convex
hull of a collectively compact set is collectively compact. Hence
we get that $\{AT_n+B\}$ and $\{T_nA+B\}$ are collectively compact.
\end{proof}

We also need the following classical result from complex analysis.

\begin{theorem}\label{BC}\cite[Theorem 3.1, Page 338]{LANG} (Borel-Caratheodory theorem).
Let $g$ be an analytic function on a disk of radius $R$ centered at
zero in the complex plane. Suppose $0<r<R.$ Then
$$ \sup _{|z|\leq r} |g(z)| \leq \frac{2r}{R-r} \sup _{|z|\leq
R}\mbox{Re }g(z)+\frac{R+r}{R-r}|g(0)|.$$

\end{theorem}

It is very important for our analysis to keep track of the
multiplicities of eigenvalues of a compact operator $A$. If $\lambda
$ is a non-zero eigenvalue of $A$, then we consider the `algebraic
multiplicity'  of $A$, which is the dimension of $E_{\lambda}(A)$,
the root-space spanned by generalized eigenvectors with eigenvalue
$\lambda :$
$$E_{\lambda }(A):=\mbox{span} \{ x: (A-\lambda )^rx=0, ~~\mbox{for some}~~ r\geq 1\}.$$
We know that this space is finite dimensional. Throughout the article, eigenvalues of
compact operators would always be considered with this
multiplicity. The theorem of Davis \cite{CHAN} is an ingredient for
our proof, and it uses this multiplicity. For self-adjoint compact
operators algebraic multiplicity is same as the geometric
multiplicity (dimension of the corresponding eigenspace).

Note that, if $\lambda$ is an isolated point of $\sigma(A)$, then we
can define the Riesz projection of $A$ corresponding to the isolated
point $\{\lambda\}$ as
\begin{equation*}
    P_{\lambda}:=\frac{1}{2\pi i}\int_{\Gamma}(\delta-A)^{-1}d\delta,
\end{equation*}
where $\Gamma$ is a closed contour around $\{\lambda\}$ separating
$\{\lambda\}$ from $\sigma(A)\backslash \{ \lambda \}$. Then
$\dim(\text{ran}P_{\lambda})$ is  equal to the algebraic
multiplicity of $\lambda$ (see Page 26, \cite{GOH} for more
details.)

We enumerate the non-zero eigenvalues of a compact operator $A$ as,
$\lambda _1, \lambda _2, \ldots $, with $|\lambda _1|\geq |\lambda
_2|\geq \cdots $, where each non-zero eigenvalue is repeated with
its algebraic multiplicity. If this collection is finite, say has
size $M$, then by convention we take $\lambda _i=0$ for $i>M.$  We
will call any such sequence as an {\em eigen-sequence} of $A$. For a compact operator $A$, the multiplicity of $0$ is defined as
the dimension of $\{x: \lim _{n\to \infty}A^nx=0\}.$ It is to be
noted that the multiplicity of $0$ may not be accounted for in the
sequence $(\lambda _1, \lambda _2, \ldots ),$ even if $0$ is an
eigenvalue. For this reason calling $(\lambda _i)$ enumerated as
above as eigen-sequence of $A$ is slight abuse of notation.

\section {Main Results}

Following Nayak \cite{SOY} for any bounded operator $T$ on $\clh$
and non-negative scalar $r$, define the following set:
\begin{align*}
        V(T,r)=\{x\in\clh:\limsup\langle|T^n|x,x\rangle^{1/n}\leq
        r\}.
\end{align*}
We have this definition for arbitrary bounded operators, but we
would be using it mostly for compact operators. Our first result
yields that the set $V(T,r)$ is a subspace for every $r\geq 0$ and
also gives a relation between the eigenspace of $T$ and $V(T,r)$.
\begin{proposition}\label{prop1}
        Let $T\in\clb(\clh)$. Then we have the following.
        \begin{enumerate}
            \item  $V(T,r)$ is a subspace of $\clh$ for $r\geq 0$.
            \item If $\lambda$ is an eigenvalue of $T$, then $\text{ ker}(T-\lambda I)\subseteq V(T,r)$ for every $r\geq|\lambda|.$
           \item   For a self adjoint operator $T$, if $\lambda$ is an isolated point of $\sigma(T)$ with $|\lambda|=r$, then
           ran$P_{\lambda}\subseteq V(T,r+\epsilon)$ for any  $\epsilon>0$, where $P_{\lambda}$ is the Riesz projection of $T$ corresponding to the isolated point ${\lambda}.$
        \end{enumerate}
\end{proposition}
    \begin{proof} Let $r\geq 0$.
        \begin{enumerate}
            \item Suppose $x,y\in V(T,r)$ and $\alpha\ne 0$ is a scalar. Then
            \begin{align*}
                \limsup\langle|T^n|\alpha x,\alpha x\rangle^{1/n}=&\limsup|\alpha|^{2/n}\langle |T^n|x,x\rangle^{1/n}\\
                \leq&\lim |\alpha|^{2/n}\limsup\langle|T^n|x,x\rangle^{1/n}\\
                \leq& r.
            \end{align*}
We also have,
            \begin{align*}
            & \limsup\langle|T^n|(x+y),x+y\rangle^{1/n}\\
                \leq&\limsup\left(\langle|T^n|x,x\rangle+\langle|T^n|x,y\rangle+\langle|T^n|y,x\rangle+\langle|T^n|y,y\rangle\right)^{1/n}\\
                \leq&\limsup\left(\langle|T^n|x,x\rangle+2\langle|T^n|x,x\rangle^{1/2}\langle|T^n|y,y\rangle^{1/2}+\langle|T^n|y,y\rangle\right)^{1/n}\\
                =&\limsup\left(\langle|T^n|x,x\rangle^{1/2}+\langle|T^n|y,y\rangle^{1/2}\right)^{2/n}\\
                \leq&\max\left\{\limsup\langle|T^n|x,x\rangle^{1/n},\limsup\langle|T^n|y,y\rangle^{1/n}\right\}\\
                \leq &~ r.
            \end{align*}
The second last inequality follows from \cite[Lemma 2.2]{SOY}. Hence
$V(T,r)$ is a subspace of $\clh$.
            \item Let $x\in\text{ ker }(T-\lambda I).$ Then
            \begin{align*}
                \langle|T^n|x,x\rangle^{1/n}\leq&\||T^n|x\|^{1/n}\|x\|^{1/n}\\
                =&\|T^nx\|^{1/n}\|x\|^{1/n}\\
                =&|\lambda|\|x\|^{2/n}.
            \end{align*}
              Hence $\limsup\langle|T^n|x,x\rangle^{1/n}\leq|\lambda|\leq r$ for every $r\geq|\lambda|$.
            \item This follows directly from \cite[Page 582, 25]{DUN}.\qedhere
        \end{enumerate}
            \end{proof}
For a positive operator $T\in\clb(\clh)$, the renowned spectral
theorem ensures that there exists a projection valued measure $E$
supported on $[0,\|T\|]$ such that
        \begin{equation*}
                T=\int_0^{\|T\|}\lambda\, dE.
            \end{equation*}
            Now, we observe a relationship between the subspaces $V(T,r)$ and the spectral projections
            of $T$.
    \begin{proposition}\label{prop positive operator subspace}
        Let $T\in\clb(\clh)$ be a positive operator. For $r\geq 0$, suppose $\lambda_r$ is the largest spectral value of $T$ less than or equal to $r$. Then
        \begin{align}\label{eqn 1}
            V(T,r)=V(T,\lambda_r)=\text{ ran}(E[0,\lambda_r]),
        \end{align}
        where $E$ is the spectral measure corresponding to the operator $T$.
    \end{proposition}
    \begin{proof}
        By spectral theorem for positive operators, we know that
        \begin{align*}
            T=\int_{0}^{\|T\|}\lambda\, dE\text{ and }T^n=\int_{0}^{\|T\|}\lambda^ndE,
~~{\mbox        for}~~ n\in\mathbb{N}.
 \end{align*}
Let $x\in\clh$ be a non-zero vector. We claim that $\limsup\langle
T^nx,x\rangle^{1/n}=\delta$ if and only if $\delta$ is the smallest
value in $[0,\|T\|]$ satisfying $\langle
E(\delta,\|T\|]x,x\rangle=0$.

Let $\delta$ be the smallest value in $[0,\|T\|]$ satisfying
$\langle E(\delta,\|T\|]x,x\rangle=0$. Then
    \begin{align*}
        \langle T^nx,x\rangle=&\int_0^{\delta}\lambda^nd \langle Ex,x\rangle\\
        \leq&\delta^n\int_0^{\delta}d\langle Ex,x\rangle\\
        \leq& \delta^n \|x\|^2.
    \end{align*}
    Consequently, $\limsup\langle T^nx,x\rangle^{1/n}\leq \delta.$

Let $\epsilon>0$ be arbitrary. Then $\langle
E(\delta-\epsilon,\delta]x, x\rangle\ne 0$, otherwise $\delta$ is
not the smallest value in $[0,\|T\|]$ satisfying  $\langle
E(\delta,\|T\|]x,x\rangle=0$. Thus
    \begin{align*}
        \langle T^nx,x\rangle=&\int_0^{\delta}\lambda^nd\langle Ex,x\rangle\\
        \geq&\int_{\delta-\epsilon}^{\delta}\lambda^nd\langle Ex,x\rangle\\
        \geq&(\delta-\epsilon)^n\langle E(\delta-\epsilon,\delta]x,x\rangle.
    \end{align*}
    Since $\lim_{n\rightarrow\infty}\langle E(\delta-\epsilon,\delta]x,x\rangle^{1/n}=1$, we have $\limsup\langle T^nx,x\rangle^{1/n}\geq \delta-\epsilon.$ Thus $$\limsup\langle T^nx,x\rangle^{1/n}\geq \delta,$$ and consequently $\limsup\langle T^nx,x\rangle^{1/n}= \delta$. If $\limsup\langle T^nx,x\rangle^{1/n}=\delta$, then arguing on similar lines as above, we get that $\langle E(\delta,\|T\|]x,x\rangle=0$.

    Let $r\geq 0$ and $\lambda_r$ be the largest spectral value of $T$ less than or equal to $r$. By definition of $V(T,r)$, we have $V(T,\lambda_r)\subseteq V(T,r).$ To get the reverse inequality, suppose $x\in V(T,r)$. By the previous claim, we know that $\limsup\langle T^nx,x\rangle^{1/n}=\tilde{\lambda}$, where $\tilde{\lambda}$ is the smallest spectral value of $T$ satisfying $\langle E(\tilde{\lambda},\|T\|]x,x\rangle=0$. Since $x\in V(T,r)$, we have $\tilde{\lambda}\leq r$. Also, $\tilde{\lambda}\leq\lambda_r$. Thus $x\in V(T,\tilde{\lambda})\subseteq V(T,\lambda_r)$. Hence $V(T,r)=V(T, \lambda_r)$.

    We know that $x\in V(T,\lambda_r)$ if and only if $\langle E(\lambda_r,\|T\|]x,x\rangle=0$. Since $E(\lambda_r,\|T\|]$ is an orthogonal projection, we have $\langle E(\lambda_r,\|T\|]x,x\rangle=0$ if and only if $x\in \text{ ran }(E[0,\lambda_r]).$ Hence $V(T,\lambda_r)=\text{ ran}(E[0,\lambda_r]).$
    \end{proof}
    In above results, for a given $T\in\clb(\clh)$  we have studied the spaces $V(T,r)$ for $r\geq 0$. Now, we go the reverse way
 and show that these spaces can be used to identify the operator.
\begin{proposition}\label{prop of uniqueness}
Let $T_1,T_2\in\clb(\clh)$ be two positive operators. If
$V(T_1,r)=V(T_2,r)$ for every $r\geq 0$, then $T_1=T_2$.
    \end{proposition}
    \begin{proof}
Let $E_1$ and $E_2$ be spectral measures of $T_1, T_2$ respectively.
We assume that $V(T_1,r)=V(T_2,r)$ for every $r\geq 0$. By
Proposition \ref{prop positive operator subspace}, we have
        $$\text{ ran}(E_1[0,a_1])=\text{ ran}(E_2[0,a_2]),$$
        where $a_1$ and $a_2$ are largest spectral values of $T_1$ and $T_2$ less than equal to $r$, respectively.
We have to show that $a_1=a_2$. On the contrary, assume that
$a_1<a_2$. If $a_1<\tilde{r}<a_2$, then $V(T_2,
\tilde{r})=V(T_2,\tilde{a}_2)$, where $\tilde{a}_2<a_2$ is a
spectral value of $T_2$. We have
\begin{align*}
            \text{ ran}(E_2[0,a_2])=\text{ ran}(E_1[0,a_1])=V(T_1,r)=V(T_1, \tilde{r})=V(T_2,\tilde{r})=\text{ ran}(E_2[0,\tilde{a}_2]),
        \end{align*}
which is a contradiction. Thus $a_1=a_2$ and consequently $\text{
ran}(E_1[0,a_1])=\text{ ran}(E_2[0,a_1])$. Hence
$E_1[0,r]=E_1[0,a_1]=E_2[0,a_1]=E_2[0,a_2]=E_2[0,r]$ for every
$r\geq 0$. This implies $T_1=T_2$.
    \end{proof}

In the following, we focus our attention on a   compact operator $A$
on a separable Hilbert space $\mathcal{H}.$ We want to show the norm
convergence of the sequence $\{ |A^n|^{\frac{1}{n}}: n\geq 1\}$. It
is often convenient to assume that $\|A\|<1.$ This is possible in
view of the following observation.

\begin{lemma}\label{scaling}
Let $A\in\clb(\clh)$ be any operator and let $c>0$ be a scalar. Then
the sequence $(|A^n|^{1/n})$ is norm convergent  if and only if $(|
(cA)^n|^{1/n})$ is norm convergent. Moreover, in such a case,
$$\lim_{n\to \infty} |(cA)^n|^{1/n}= |c|\lim _{n\to
\infty}|A^n|^{1/n}.$$
\end{lemma}
\begin{proof}
This is clear, as $|(cA)^n|^{1/n}=|c| |A^n|^{1/n}$, for $n\in\mathbb{N}$.
\end{proof}

Consider
    \begin{align*}
        \zeta _A: \text{ set of all norm limit points of }(|A^n|^{1/n}).
    \end{align*}

Our first task is to show that $\zeta _A$ is non-empty and consists
of positive compact operators. This is accomplished by making use of
the notion of collective compactness.  Subsequently, we will show
that all limit points of $ \{|A^n|^{1/n}:n\geq 1\}$ are equal, that
is, $\zeta _A $ is singleton and the whole sequence converges to
this limit point.

\begin{theorem}\label{collectively compact}
Let $A\in\clb(\clh)$ be a compact operator with $\|A\|<1$. Then the
set $\{|A^n|^{1/n}:n\in\mathbb{N}\}$ is collectively compact.
    \end{theorem}
    \begin{proof}
    Let $\{e_i:i\in\mathbb{N}\}$ be an orthonormal basis for the Hilbert space $\clh$. Then $A$ admits an infinite matrix representation $(a_{ij})$ with respect to the basis. Define
    \[\tilde{A}_m=(a_{ij})_{i,j=1}^m\text{ and }A_m=\begin{bmatrix}
        \tilde{A}_m&0\\
        0&0
    \end{bmatrix}.\]
    As in the proof of Theorem $2$ of \cite{YAMA2}, we have $(A_m)$ converges to $A$ in the norm topology.
Then it follows  that $(A_m^n)$ converges in norm topology to $A^n$
for every $n.$  We claim that this convergence takes place uniformly
in $n$.

Since $\|A\|<1$, we have $\|A_m\|<1.$ Using analytic functional
calculus, we know that
    \begin{align*}
        \|A_m^n-A^n\|=&\frac{1}{2\pi i}\left\|\int_{\mathbb{T}}\lambda^n\left[(\lambda-A_m)^{-1}-(\lambda-A)^{-1}\right]d\lambda\right\|\\
        \leq&\underset{\lambda\in \mathbb {T} }{\sup}|\lambda|^n\|(\lambda-A_m)^{-1}-(\lambda-A)^{-1}\|\\
        =&\underset{\lambda\in \mathbb{T}
        }{\sup}\|(\lambda-A_m)^{-1}\|\|A_m-A\|\|(\lambda-A)^{-1}\| ,
    \end{align*}
where $\mathbb{T} $ is the unit circle centered at the origin with
radius $1$. We observe  that $\|A_m\|\leq\|A\|$ and consequently,
    \begin{align*}
        \|(\lambda-A_m)^{-1}\|=\|(1-\lambda^{-1}A_m)^{-1}\|\leq\frac{1}{1-\|A_m\|}\leq\frac{1}{1-\|A\|},
    \end{align*}
    for every $\lambda\in \mathbb{T} $. Now, we have
    $$\|A_m^n-A^n\|\leq\frac{1}{(1-\|A\|)^2}\|A_m-A\|\text{ for every }n\in\mathbb{N},$$
    which implies that $(A_m^n)$ converges to $A^n$ in norm topology, uniformly in
    $n$, as $m$ tends to $\infty .$ This proves our claim.

Now for $n\in \mathbb{N}$,
\begin{align*}
        \|A_m^{*n}A_m^n-A^{*n}A^n\|\leq&\|A_m^{*n}A_m^n-A^{*n}A_m^n\|+\|A^{*n}A_m^n-A^{*n}A^n\|\\
        \leq&\|A_m^{*n}-{A}^{*n}\|+\|A_m^n-A^n\|\\
        \leq&2\|A_m^n-A^n\|.
    \end{align*}
From the  inequality above, we conclude that
$\|A_m^{*n}A_m^n-A^{*n}A^n\|$ converges to $0$ uniformly in $n$.
Consider any analytic polynomial $p$. Using analytic
functional calculus, we have
    \begin{align*}
         \|p(A_m^{*n}A_m^n)-p(A^{*n}A^n)\|=&\frac{1}{2\pi i}\left\|\int_{\mathbb{T} }p(\lambda)[(\lambda-A_m^{*n}A_m^n)^{-1}-(\lambda-A^{*n}A^n)^{-1}]d\lambda\right\|\\
         \leq&\|p\|_{\infty, \mathbb{D}}\|(\lambda-A_m^{*n}A_m^n)^{-1}\|\|A_m^{*n}A_m^n-A^{*n}A^n\|\|(\lambda-A^{*n}A^n)^{-1}\|\\
         \leq&\frac{\|p\|_{\infty ,
         \mathbb{D}}}{(1-\|A\|)^2}\|A_m^{*n}A_m^n-A^{*n}A^n\|,
    \end{align*}
where $\|p\|_{\infty , \mathbb{D}}:=\sup \{|p(z)|: |z|<1\}.$ In the
above computations, the last inequality follows from the fact that
$\|{A^*}^nA^n\|\leq\|A\|^{2n}\leq\|A\|$ for $n\in\mathbb{N}$ and
using the following inequality.
    \begin{align*}
        \|(\lambda-{A^*}^nA^n)^{-1}\|=\|(1-\lambda^{-1} {A^*}^nA^n)^{-1}\|\leq\frac{1}{1-\|{A^*}^nA^n\|}\leq\frac{1}{1-\|A\|},\text{ for }\lambda\in\mathbb{T}.
    \end{align*}
A similar calculation implies that
$\|(\lambda-A_m^{*n}A_m^n)^{-1}\|\leq\frac{1}{1-\|A\|},$ for
$\lambda\in\mathbb{T}$. Hence,
\begin{equation}\label{bound} \|p(A_m^{*n}A_m^n)-p(A^{*n}A^n)\|\leq
\frac{\|p\|_{\infty ,
\mathbb{D}}}{(1-\|A\|)^2}\|A_m^{*n}A_m^n-A^{*n}A^n\|.
\end{equation}

%From these computations, we can conclude that
%$\|p(A_m^{*n}A_m^n)-p(A^{*n}A^n)\|$ converges to $0$ uniformly in
%$n$ and $p\in\mathcal{M}_k$.

Now, we claim that
$\left\|(A_m^{*n}A_m^n)^{1/2n}-(A^{*n}A^n)^{1/2n}\right\|$ converges
to $0$ uniformly in $n$.  To prove this, fix $0<\epsilon  <1$. We
define $f_n(z)=|z|^{1/2n}$ for $z\in 2\overline{\mathbb{D}},$ where
$2\overline{\mathbb{D}}$ is the closed unit disc of radius $2$,
around the origin. Clearly, it is a continuous function on a compact
Hausdorff space.  Then the Stone-Weierstrass approximation theorem
implies that there exists a trigonometric polynomial,  $p_n$ such
that $\|f_n-p_n\|_{\infty,2\mathbb{D}}:=\sup\{|(f_n-p_n)(z)|:
|z|<2\} <\frac{\epsilon}{4}$.
% Since $\|f_n\|_{\infty,2\mathbb{D}}\leq 2$ for every $n\in\mathbb{N}$, we
%have $\|p_n\|_{\infty,2\mathbb{D}}<2+\frac{\epsilon}{4}$ for every
%$n\in\mathbb{N}$.}
 Note that $p_n$ is a trigonometric polynomial in
$z$, means that it is of the form, $a_n(z)+b_n(\bar{z})$, where
$a_n, b_n$ are usual polynomials. Define $p_n^{(j)}$ for $j=1,2,3,
4$ by
$$p_n^{(1)}(z)=p_n(z);\, p_n^{(2)}(z)=p_n(\bar{z});\,
p_n^{(3)}(z)=\overline {p_n(z)};\,
p_n^{(4)}(z)=\overline{p_n(\bar{z})}.$$  Observe that for $z\in
2\overline{\mathbb{D}}$,
$$f_n(z)=f_n(\bar{z})=\overline {f_n(z)}= \overline{f_n(\bar{z})}.$$
 It follows that
$\|f_n-p_n\|_{\infty,2\mathbb{D}}=\|f_n-p_n^{(j)}\|_{\infty,2\mathbb{D}}$
for $1\leq j\leq 4$. Consequently, taking $q_n= \frac{1}{4}\sum
_{j=1}^4p_n^{(j)}$,
$$\|f_n-q_n\|_{\infty,2\mathbb{D}}=\|f_n-p_n\|_{\infty,2\mathbb{D}}< \frac{\epsilon }{4}.$$
It is clear that $\|f_n\|_{\infty , 2\mathbb{D}}=2.$ Hence
$$\|q_n\|_{\infty , 2\mathbb{D}}<2+\frac{\epsilon}{4}.$$

From the definition of $q_n$, it is clear that it is of the form
$q_n(z)= g_n(z)+\overline{g_n(z)}$, where $g_n$ is a polynomial with
real coefficients. Then by Borel-Caratheodory Theorem (see Theorem
\ref{BC}) for $r=1$ and $R=2$,
$$\sup _{|z|\leq 1}|g_n(z)|\leq 2\sup_{|z|\leq 2}q_n(z)+
3|g_n(0)|.$$ Observe that
$q_n(0)=g_n(0)+\overline{g_n(0)}=2g_n(0)$. Hence
$|g_n(0)|=\frac{1}{2}|q_n(0)|\leq
\frac{1}{2}\|q_n\|_{\infty,2\mathbb{D}}<1+\frac{\epsilon}{8}.$
Consequently, we have   $\|g_n\|_{\infty,\mathbb{D}}<
2(2+\frac{\epsilon }{4})+ 3(1+\frac{\epsilon}{8})<8$.
%But $p_n(A_m^{*n}A_m^n)$ is equal to $q_n(A_m^{*n}A_m^n)$ for some
%analytic polynomial $q_n$ in $\mathbb{D}$, as $A_m^{*n}A_m^n$ is
%self-adjoint. By continuous functional calculus for self-adjoint
%operators, we know that
%$\|p_n(z,\overline{z})\|_{\infty,D}=\|p_n(A_m^{*n}A_m^n,A_m^{*n}A_m^n)\|=\|q_n(A_m^{*n}A_m^n)\%|=\|q_n\|_{\infty,D}$.
%Thus $\|q_n\|_{\infty,D}<1+\frac{\epsilon}{3}$ for $n\in\mathbb{N}$.

Note that $A_m^{*n}A_m^n$ and $A^{*n}A^n$ are self-adjoint operators
for $m,n\in\mathbb{N}$. By continuous functional calculus for
self-adjoint operators, we have
$q_n(A_m^{*n}A_m^n)=2g_n(A_m^{*n}A_m^n)$ and
$q_n(A^{*n}A^n)=2g_n(A^{*n}A^n)$. Then by equation (\ref{bound}),
    \begin{align*}
        \left\|(A_m^{*n}A_m^n)^{1/2n}-(A^{*n}A^n)^{1/2n}\right\|=&\|f_n(A_m^{*n}A_m^n)-f_n(A^{*n}A^n)\|\\
        \leq& \|f_n(A_m^{*n}A_m^n)-q_n(A_m^{*n}A_m^n)\|+2\|g_n(A_m^{*n}A_m^n)-g_n(A^{*n}A^n)\|\\
        &+\|q_n(A^{*n}A^n)-f_n(A^{*n}A^n)\|\\
        \leq &2\|f_n-q_n\|_{\infty,\mathbb{D}}+2\frac{\|g_n\|_{\infty,
        \mathbb{D}}}{(1-\|A\|)^2}\|A_m^{*n}A_m^n-A^{*n}A^n\|\\
        \leq & 2. \frac{\epsilon}{4}+ \frac{16}{(1-\|A\|)^2}\|A_m^{*n}A_m^n-A^{*n}A^n\|.
    \end{align*}
Thus $\|(A_m^{*n}A_m^n)^{1/2n}-(A^{*n}A^n)^{1/2n}\|$ converges to
$0$ uniformly in $n$. In other words, $\{ |A_m^n|^{\frac{1}{n}}\}$
converges to $|A^n|^{\frac{1}{n}}$, uniformly in $n$ as $m$ tends to
$\infty .$

By Nayak's theorem for matrices  \cite{SOY}, we know that
$(|A_m^n|^{1/n})$ converges to some positive finite rank operator
$B_m$. Consequently, Proposition \ref{Palmer 2.2} gives that
$\{|A_m^n|^{1/n}-B_m:n\in\mathbb{N}\}$ is collectively compact for
each $m$. Then Proposition \ref{Palmer 2.1} implies that
$\{|A_m^n|^{1/n}:n\in\mathbb{N}\}$ is collectively compact for each
$m$ and hence Proposition \ref{Ane} implies that
$\{|A^n|^{1/n}:n\in\mathbb{N}\}$ is collectively compact.
\end{proof}

\begin{lemma}\label{prop limit points of comapct}
Let $A\in\clb(\clh)$ be a compact operator with $\|A\|< 1$. Then the
set $\zeta _A$ of norm limit points of $(|A^n|^{1/n})$ is non-empty
and consists of positive compact operators.
\end{lemma}
    \begin{proof} Let $A\in\clb(\clh)$ be a compact operator. Then
$\||A^n|^{1/n}\|\leq\|A\|$ for $n\in\mathbb{N}$. By \cite[Theorem
5.1.3, Page 306]{KADI} $(|A^n|^{1/n})$ has a WOT convergent
subsequence, say $(|A^{n_k}|^{1/n_k})$ converging to $B$ in WOT.

We claim that $(|A^{n_k}|^{1/n_k})$ converges to $B$ in norm. Since
$|A^{n_k}|^{1/n_k}$ converges to $B$ in WOT, we have
$|A|(|A^{n_k}|^{1/n_k})$ converges to $|A|B$ in SOT by \cite[Page
81]{CON}.  By Theorem \ref{collectively compact}, we know that $\{|A^{n_k}|^{\frac{1}{n_k}}\}$ is a collectively compact set, %Note that, \cite[Proposition 2.3]{PALMER1} implies that $\{|A||A^{n_k}|^{\frac{1}{n_k}}\}$ and $\{|A^{n_k}|^{\frac{1}{n_k}}|A|\}$ are two collectively compact sets of operators. We know that $|A|B$, $B|A|$ are compact operators and finite union of collectively compact sets are collectively compact, thus $\{|A||A^{n_k}|^{\frac{1}{n_k}},|A|B,-|A|B,0\}$ and $\{|A^{n_k}|^{\frac{1}{n_k}}|A|,B|A|,-B|A|,0\}$ are collectively compact. By \cite[Proposition 2.1]{PALMER1}, we get that
    %\begin{equation*}
    %       \left\{\frac{|A||A^{n_k}|^{\frac{1}{n_k}}}{2}-\frac{|A|B}{2}\right\} \text{ and }   \left\{\frac{|A^{n_k}|^{\frac{1}{n_k}}|A|}{2}-\frac{B|A|}{2}\right\}
%   \end{equation*}
thus Proposition \ref{proposition linear combination}  implies that
$\{ |A||A^{n_k}|^{\frac{1}{n_k}}-|A|B\}$ and
$\{|A^{n_k}|^{\frac{1}{n_k}}|A|-B|A|\}$ are collectively compact
sets of operators. This along with the fact that
$\{|A||A^{n_k}|^{\frac{1}{n_k}}\}$ converges to $|A|B$ in SOT
implies that $\{|A||A^{n_k}|^{\frac{1}{n_k}}\}$ converges to $|A|B$
in norm, by part (3) of Proposition \ref{Palmer 2.3}. Taking the
adjoint, we get that $\{|A^{n_k}|^{\frac{1}{n_k}}|A|\}$ converges to
$B|A|$ in norm. 
% and$|A|(|A^{n_k}|^{1/n_k})$ converges to $|A|B$ in norm by \cite[Page
%81]{CON}. By taking adjoint, we get that $(|A^{n_k}|^{1/n_k})|A|$
%converges to $B|A|$ in norm.

        Let $x\in \text{ker}(|A|)$. Then $|A^{n_k}|^{1/n_k}x=0$ for every $k\in\mathbb{N}$ and consequently $\langle Bx,x\rangle=0$. Since $B$ is a positive operator, we have $Bx=0$. As a result, we get that $\text{ker}(|A|)\subseteq \text{ker}(B)$.

        Let $y\in\clh$, and $y=y_1+y_2$, where $y_1\in \text{ ker}(|A|)$ and $y_2\in \overline{\text{ran}}(|A|)$. For $\epsilon>0$, there exist $z_2\in\clh$ such that $\|y_2-|A|z_2\|<\frac{\epsilon}{3}.$ Then
        \begin{align*}
            \||A^{n_k}|^{1/n_k}y-By\|=&\||A^{n_k}|^{1/n_k}(y_1+y_2)-B(y_1+y_2)\|\\
            =&\||A^{n_k}|^{1/n_k}y_2-By_2\|\\
            \leq&\||A^{n_k}|^{1/{n_k}}y_2-|A^{n_k}|^{1/{n_k}}|A|z_2\|+\||A^{n_k}|^{1/{n_k}}|A|z_2-B|A|z_2\|\\
            +&\|B|A|z_2-By_2\|\\
            =&\frac{2\epsilon}{3}+\||A^{n_k}|^{1/n_k}|A|z_2-B|A|z_2\|.%\text{ where }y_2=|A|z_2.
        \end{align*}
        We know that $(|A^{n_k}|^{1/n_k}|A|)$ converges to $B|A|$ in norm, thus by above calculation, we get $|A^{n_k}|^{1/n_k}$ converges to $B$ in SOT.

By Theorem \ref{collectively compact}, we know that
$\{|A^{n_k}|^{1/{n_k}}\}$ is collectively compact and Proposition
\ref{Palmer 2.1} (2),  implies that $\{|A^{n_k}|^{1/{n_k}}-B\}$ is
collectively compact and $B$ is compact. Now, Proposition
\ref{Palmer 2.3}
 gives that $\{|A^{n_k}|^{1/{n_k}}\}$ converges to $B$ in
the norm topology.
        % By \cite[Page 81]{CON}, we know that $|A^{n_k}|^{1/n_k}|A|$ converges to $B|A|$ in operator norm. But on ker$(|A|)$, both $|A^{n_k}|^{1/n_k}$ and $B$ are zero. As a result, $|A^{n_k}|^{1/n_k}$ converges to $B$ in operator norm. Since norm limit of positive compact operators is a positive compact operator, thus $B$ is a positive compact operator. Hence the set of limit points of $(|A^n|^{1/n})$ is non-empty and consists of positive compact operator.
    \end{proof}

%%%%%%%%%%%%%%%%%%%%%%new

%%%%%%%%%%%%%%55
%\section{Main results}

\begin{lemma}\label{lemma containment of subspaces}
Let $A\in\clb(\clh)$ be a compact operator with $\|A\|<1$. Then for
 $B\in \zeta _A,$
\begin{align*}
            V(A,r)\subseteq V(B,r)\text{ for }r\geq 0.
        \end{align*}
    \end{lemma}
    \begin{proof}
As $B$ is a limit point of $(|A^n|^{1/n})$, there exists a
subsequence, say $(|A^{n_k}|^{1/n_k})$ converging to $B$. Suppose
$0\neq x\in V(A,r)$ for $r\geq 0$ and $p\in\mathbb{N}$.
Then
        \begin{align*}
            \langle B^px,x\rangle=\lim_k \langle |A^{n_k}|^{p/n_k}x,x\rangle.
        \end{align*}
        If we choose $n_k>p$, then by Proposition \ref{thm inequ furuta}, we have
            \begin{align*}
            \langle B^px,x\rangle\leq&\lim_k \langle |A^{n_k}|x,x\rangle^{p/n_k}\frac{\|x\|^2}{\|x\|^{\frac{2p}{n_k}}}\\
            \leq&\limsup_n\langle|A^n|x,x\rangle^{p/n}\|x\|^2\\
            \leq& r^p\|x\|^2.
        \end{align*}
        Thus $\limsup\langle B^px,x\rangle^{1/p}\leq r.$ Hence $x\in V(B,r).$
    \end{proof}
    In the following we see that the spectrum of operators in $\zeta _A$ is
determined by the spectrum of $A$.
    \begin{lemma}\label{lemma spectrum of limit point}
        Let $A\in\clb(\clh)$ be a compact operator with $\|A\|<1$, and let $B\in\zeta_A$. Then
        \begin{align*}
            \sigma(B)=\{|\lambda|:\lambda\in\sigma(A)\}.
        \end{align*}
  For $\lambda \neq 0$,    $\dim\text{ker}(B-|\lambda|I)$ %=\underset{\eta\in\sigma(A),\\|\eta|=|\lambda|}{\sum}\dim\text{ker}(A-\eta I)$ for $\lambda\ne 0$.
        is equal to the sum of the algebraic multiplicity of $\eta\in\sigma(A)$ for which $|\eta|=|\lambda|$.
    \end{lemma}
    \begin{proof}
There is a subsequence $(|A^{n_l}|^{1/n_l})$ converging to $B$.
        %\begin{align*}
        %    \||A^{n_l}|^{1/n_l}-B\|\rightarrow 0\text{ as }l\rightarrow\infty.
        %\end{align*}
Let $(s_i^{(n)})$ and $(\delta_i)$ be eigensequences of
$|A^{n}|^{1/n}$ and $B$, respectively. Observe that as these are
positive operators eigen-sequences are uniquely determined.
%These are arranged in such a
%way that
%        \begin{align*}
%            &s_1(A^{n_l})^{1/n_l}\geq s_2(A^{n_l})^{1/n_l}\geq\ldots,\\
%            &\delta_1\geq\delta_2\geq\ldots
%        \end{align*}
By \cite[(b), Page 123]{GOH}, we know that
        \begin{align*}
            \left|s_i^{(n_l)}-\delta_i\right|\leq  \||A^{n_l}|^{1/n_l}-B\|,\text{ for }i\in\mathbb{N}.
        \end{align*}
This implies $s_i^{(n_l)}$ converges to $\delta_i$ as
$l\rightarrow\infty$ for every $i\in\mathbb{N}.$ By \cite{CHAN}, we
know that
        \begin{align*}
            s_i^{(n)}\rightarrow |\lambda_i| \text{ as }n\rightarrow\infty,
        \end{align*}
for every $i\in\mathbb{N}$.
% where $(s_i(A^{n})^{1/n}$ and
%$(\lambda_i)$ are eigenvalues (repeated according to multiplicity)
%of $|A^{n}|^{1/n}$ and $A$, respectively, arranged in such a way
%that
%            \begin{align*}
%                &s_1(A^{n})^{1/n}\geq s_2(A^{n})^{1/n}\geq\ldots,\\
%                &|\lambda_1|\geq|\lambda_2|\geq\ldots.
%            \end{align*}
As a result, we get $\delta_i=|\lambda_i|$ for every
$i\in\mathbb{N}.$  Since $B, A$ are compact operators, $0$ is in
their spectrum anyway. Hence $
\sigma(B)=\{|\lambda|:\lambda\in\sigma(A)\}.$

Note that the multiplicity of any $\delta_i$ is same as the sum of
the algebraic multiplicities of $\lambda_j$ for which
$|\lambda_j|=\delta_i$. As a result
$\dim\text{ker}(B-\delta_iI)$ %=\underset{\lambda_j\in\sigma (A)\\ |\lambda_j|=\delta_i}{\sum}\text{ker}(A-\lambda_jI).$$
is equal to sum of the algebraic multiplicity of
$\lambda_j\in\sigma(A)$ for which
            $|\lambda_j|=\delta_i$. This proves the second part.
    \end{proof}
A simple consequence of Lemma \ref{lemma spectrum of limit point} is
that the only limit point of $(|A^n|^{1/n})$ is singleton zero for a
compact quasinilpotent operator, for example Volterra operator.
\begin{lemma}
    Let $A\in\clb(\clh)$ be a compact operator with $\sigma(A)=\{0\}$. Then $\zeta _A=\{0\}$.
%\end{corollary}
%\begin{remark}\label{remark1}
    Moreover,  $\limsup\langle|A^n|x,x\rangle^{1/n}=0$ for every $x\in\clh.$
\end{lemma}

\begin{proof}
The first part is clear from the previous result. The second part
follows from the following computations.
    For non-zero $x\in\clh$, we have
    \begin{align*}
        \limsup\langle|A^n|x,x\rangle^{1/n}\leq&\limsup\||A^n|x\|^{1/n}\|x\|^{1/n}\\
        =&\limsup\|A^nx\|^{1/n}\|x\|^{1/n}\\
        \leq&\limsup\|A^n\|^{1/n}\|x\|^{2/n}=0.
    \end{align*}
    The last inequality holds because spectral radius of $A$ is zero.
\end{proof}

   % \begin{theorem}\label{thm limit point of compact op and spaces}
     %   Let $A\in\clb(\clh)$ be a compact operator with $\|A\|< 1$.
%    \end{theorem}
 Finally we are in a position to state and prove our
main theorem.

\begin{theorem}\label{main result norm1}
Let $A\in\clb(\clh)$ be a compact operator. Then the sequence
$(|A^n|^{1/n})$ converges in norm to a positive compact operator
$B$, where for $r\geq 0$,
        $$\text{ker}(B-r I)=\{x\in\clh:\limsup\langle|A^n|x,x\rangle^{1/n}=r\}\bigcup \{0\}.$$
%Moreover, for $r>0$, $\text{ker}(B-rI)$ is finite dimensional and
%equals $\text{span}\{x\in \clh : (A-\lambda I)^kx=0, ~~\mbox{for
%some} k\geq 1, \lambda \in \mathbb{C}~\mbox{with} |\lambda |=r\}.$
Further, $(|A^n|)^{\frac{1}{n}})$ converges to $0$ if and only if
$A$ is quasi-nilpotent, that is, $\sigma(A)=\{0\}.$
\end{theorem}
    \begin{proof}
Making use of Lemma \ref{scaling}, in the following we assume
  $\|A\|\leq 1$.

First we will show that if  $B\in\zeta _A$, then
        \begin{align*}
            V(A,r)=V(B,r)\text{ for }r\geq 0.
        \end{align*}
By Proposition \ref{prop of uniqueness}, this will show that $\zeta
_A$ is singleton.
 By
Lemma \ref{lemma containment of subspaces}, we already know that
$V(A,r)\subseteq V(B,r)$ for every $r\geq 0$.  Let $(\lambda_i)$ be
an eigen-sequence of  $A$. Note that for non-zero $x\in\clh$, we
have
$$\langle|A^n|x,x\rangle^{1/n}\leq\||A^n|x\|^{1/n}\|x\|^{1/n}\leq\||A^n|\|^{1/n}\|x\|^{2/n}=s_1(A^n)^{1/n}\|x\|^{2/n}.$$
By \cite{CHAN}, we get that
$\limsup\langle|A^n|x,x\rangle^{1/n}\leq|\lambda_1|$ for non-zero
vector $x\in\clh$. As a result,
$V(B,|\lambda_1|)=\clh=V(A,|\lambda_1|),$ where the first equality
follows from Proposition \ref{prop positive operator subspace}.
Suppose $|\lambda_1|=|\lambda_2|=\cdots=|\lambda_{i_1}|$ and
$|\lambda_{i_1+1}|<|\lambda_{i_1}|$. Then
$V(B,|\lambda_j|)=V(B,|\lambda_1|)=\clh=V(A,|\lambda_1|)=V(A,|\lambda_j|)$
for $1\leq j\leq i_1$.

    Now, assume that $|\lambda_{i_1+1}|=|\lambda_{i_1+2}|=\cdots=|\lambda_{i_2}|$.
We claim that $V(B,|\lambda_{i_2}|)=V(A,|\lambda_{i_2}|)$.  Consider a real number $r$ such that $|\lambda_{i_1+1}|<r<|\lambda_{i_1}|.$     Let $P_{r}$ be the Riesz projection corresponding to the operator $A$ defined by
$$P_{r}=\frac{1}{2\pi i}\int_{\Gamma_r}(\mu-A)^{-1}d\mu,$$
    where integration is taken over the circle centered at the origin with radius $r$. By \cite[Page 582,25]{DUN}, we have
    $$\limsup\langle|A^n|x,x\rangle^{1/n}\leq\limsup\|A^nx\|^{1/n}<r\text{ for }x\in\text{ran}P_r.$$
Hence $\text{ran}P_r\subseteq V(A,r)$ and consequently $\dim V(A,r)^{\perp}\leq\dim\text{ker}P_r=i_1$.
Note that as $B$ is a positive compact operator, $\clh=\text{ker}B \oplus \underset{i=1, \lambda _i\neq 0}{\overset{\infty}{\oplus}}\text{ker}(B-|\lambda_i|I)$ and %for $\lambda_i\ne 0$, we have $\dim\text{ker}(B-|\lambda_i|I)=\underset{k}{\sum}\dim\text{ker}(A-\lambda_{i_k}I)$, where sum is taken over all $k$ for which $|\lambda_{i_k}|=|\lambda_i|$.
by Proposition \ref{prop positive operator subspace}, we have
    \begin{align*}
    \dim(V(B,r))^{\perp}=   \dim( V(B,|\lambda_{i_2}|)^{\perp})=\dim \text{ker}(B-|\lambda_1|I)=i_1\geq\dim( V(A,r)^{\perp}).
    \end{align*}
    %The last inequality follows from Equation (\ref{eqn2}), which gives that $\text{rang}E_{\lambda_1}\supseteq V(A,|\lambda_{i_2}|)^{\perp}.$
    By Lemma \ref{lemma containment of subspaces}, we know that $V(A,r)\subseteq V(B,r)$ and consequently $$\dim( V(B,r)^{\perp})\leq\dim( V(A,r)^{\perp}).$$
This implies that $\dim( V(B,r)^{\perp})=\dim( V(A,r)^{\perp}).$
 Now from  $V(A,r)\subseteq V(B,r)$, we get
$V(B,r)=V(A,r)$. Since $|\lambda_k|=|\lambda_{i_2}| $ for $i_1+1\leq
k\leq i_2$, we have
$V(B,r)=V(B,|\lambda_{k}|)=V(B,|\lambda_{i_2}|)=V(A,r)\supseteq
V(A,|\lambda_{i_2}|)=V(A,|\lambda_{k}|)$ for $i_1+1\leq k\leq i_2$.
Let $x\in V(B,|\lambda_{i_2}|)$. Then $x\in V(B,r)$ for every
$|\lambda_{i_1+1}|<r<|\lambda_{i_1}|$, as a result $x\in V(A,r)$.
Taking limit $r$ going to $|\lambda_{i_2}|$, we get that $x\in
V(A,|\lambda_{i_2}|)$ and this proves our claim that
$V(B,|\lambda_{i_2}|)=V(A,|\lambda_{i_2}|).$  Continuing the same
procedure, we get that $V(A, |\lambda|)=V(B,|\lambda|)$ for every
non-zero $\lambda$. Also, observe that
$\text{ker}(B-|\lambda|I)=\{x\in\clh:\limsup\langle|A^n|x,x\rangle^{1/n}=|\lambda|\}\bigcup
\{0\}$.

 We claim that, if $\lambda=0$, then $V(B,0)=\text{ ker}B=\{x\in\clh:\limsup\langle|A^n|x,x\rangle^{1/n}=0\}=V(A,0).$ If $\limsup\langle|A^n|x,x\rangle^{1/n}=0$, then $\liminf\langle|A^n|x,x\rangle^{1/n}=0$ and as a result the following limit exists and $$\lim\langle|A^n|x,x\rangle^{1/n}=0.$$
 By Proposition \ref{thm inequ furuta}, $\langle
Bx,x\rangle=\lim\langle|A^{n_k}|^{1/{n_k}}x,x\rangle=0$.  Hence
$x\in \text{ ker}(B)$. On the other hand, if $y\in\text{ ker}(B)$
and $\limsup\langle|A^n|y,y\rangle^{1/n}=|\lambda|\ne
0$, then we know that $x\in \text{ ker}(B-|\lambda|I)$, which is a
contradiction.

%    Suppose $|\lambda_{i+1}|<r<|\lambda_{i}|$ for some $i$. Then by Proposition \ref{prop positive operator subspace}, we know that $V(B,r)=V(B,|\lambda_{i+1}|)$. If $x\in V(A,r)$, then Lemma \ref{lemma containment of subspaces} implies that $x\in V(B,r)=V(B, |\lambda_{i+1}|)$ and consequently $x\in V(A,|\lambda_{i+1}|)$. This proves that $V(A,r)=V(A,|\lambda_{i+1}|)=V(B,|\lambda_{i+1}|)=V(B,r).$
    %\end{proof}

%%%%%%%%%%%%%%%%%%

Suppose $B_1, B_2\in\zeta_A$. Then
    \begin{align*}
        V(B_1,r)=V(A,r)=V(B_2,r),\,r\geq 0.
    \end{align*}
Proposition \ref{prop of uniqueness} implies that $B_1=B_2$. Thus
$(|A^n|^{1/n})$ has a unique norm limit point, say $B$. We know that
any WOT limit point of this sequence is also a norm limit point.
From the sequential WOT compactness of the unit ball of
$\mathscr{B}(\mathcal{H})$, it follows that $(|A^n|^{1/n})$
converges to $B$. By Lemma \ref{lemma spectrum of limit point} and
the spectral theorem for positive compact operators, we have
    \begin{align*}
        B=\underset{\delta =|\lambda |,~\mbox {for some} ~\lambda \in\sigma(A)}{\sum}\delta  E_\delta,
    \end{align*}
where $E_\delta$ is orthogonal projection onto $\text{ker}(B-\delta
I)=\{x\in\clh:\limsup\langle|A^n|x,x\rangle^{1/n}=\delta
\}\cup\{0\}.$
    \end{proof}

If we drop the condition that the operator is compact then the above result does not hold. The following example illustrates this.
\begin{example}\label{example1}
    Let $L$ be the left shift operator on $\ell^2(\mathbb{N})$ defined by
    \begin{align*}
        L(x_1,x_2,\ldots)=(x_2,x_3,\ldots)\text{ for }(x_1,x_2,\ldots)\in\ell^2(\mathbb{N}).
    \end{align*}
    Then $L^{*n}L^n=P_n$, where $P_n$ is the orthogonal projection on the space $\overline{\text{span}}\{e_i:i>n\}$, where $\{e_m:m\in\mathbb{N}\}$ is the standard orthonormal basis for $\ell^2(\mathbb{N})$. Consequently $(|L^n|^{1/n})=(P_n)$ converges to zero in SOT but not in the norm topology.
\end{example}
\begin{example}\label{example2}
    Consider an operator $T$  on $\ell^2(\mathbb{N})$ defined on the standard basis $(e_i)_{i\in\mathbb{N}}$ of $\ell^2(\mathbb{N})$ by
    \begin{align*}
        Te_i=\begin{cases}
            \frac{e_2}{2}&\text{ if }i=1;\\
            \frac{e_{2^n+1}}{2^{2^n-1}}&\text{ if }i=2^n;\\
            2e_{i+1}&\text{ otherwise}.
        \end{cases}
        %Te_1=\frac{e_2}{2},\,Te_{2^n}=\frac{e_{2^n+1}}{2^{2^n-1}},\,Te_i=2e_{i+1}\text{ otherwise.}
    \end{align*}
    Note that $T$ is a weighted shift with weights less than or equal to $2$ and hence $T$ is a  bounded linear operator.
To have more clarity about the operator $T$, we give an expended
form of the operator $T$ and its adjoint:
    \begin{align*}
    T(x_1,x_2,\ldots)=&\left(0,\frac{x_1}{2},\frac{x_2}{2},2x_3, \frac{x_4}{2^3},2x_5,2x_6,2x_7,\frac{x_8}{2^7},2x_9,\ldots\right),\\
    T^*(x_1,x_2,\ldots)=&\left(\frac{x_2}{2},\frac{x_3}{2},2x_4,\frac{x_5}{2^3},2x_6,2x_7,2x_8,\frac{x_9}{2^7},2x_{10},\ldots\right).
        \end{align*}
     Observe that $(T^*)^nT^n$ is diagonal for every $n$. Then by simple computations, it follows that
    \begin{align*}
        |T|e_1=\frac{e_1}{2},\, |T^2|^{1/2}e_1=\frac{e_1}{2},\,
        |T^3|^{1/3}e_1=\frac{e_1}{2^{1/3}}, \,  |T^4|^{1/4}e_1=\frac{e_1}{2},
    \end{align*} and for $n\geq 2,$
    \begin{align*}
        |T^{2^n-1}|^{\frac{1}{2^n-1}}e_1=\frac{e_1}{2^{\frac{1}{2^n-1}}},\, |T^{2^n}|^{\frac{1}{2^n}}e_1=\frac{e_1}{2}.
    \end{align*}
    As a result
    \begin{align*}
        \langle |T^{2^n-1}|^{\frac{1}{2^n-1}}e_1,e_1\rangle&=\frac{1}{2^{\frac{1}{2^n-1}}}\rightarrow 1,\\
        \langle |T^{2^n}|^{\frac{1}{2^n}}e_1,e_1\rangle&=\frac{1}{2}\rightarrow\frac{1}{2}.
    \end{align*}
In other words the sequence $(\langle|T^m|^{1/m}e_1,e_1\rangle)$ has
two subsequences converging to two different points, and hence the
sequence is not convergent. Consequently,  $(|T^m|^{1/m})$ is not
convergent in WOT.
\end{example}

\section*{Acknowledgments}
Bhat gratefully acknowledges funding from  SERB (India) through J C
Bose Fellowship No. JBR/2021/000024. In an earlier
version of this article to prove Lemma \ref{prop limit points of
comapct}, we had used Exercise  5 of \cite[Page 81]{CON}, which
claims that if a sequence of operators $\{T_n\}$ converges to 0 in
SOT and $K$ is compact then $\{KT_n\}$  converges to 0 in norm. We
thank  S. Nayak and R. Shekhawat for pointing out  that this is
false.

\end{document}